\documentclass[12pt,a4paper]{article}
\usepackage[centertags]{amsmath}
\usepackage{amsfonts}

\usepackage{graphics}
\usepackage{amssymb}
\usepackage{amsthm}
\usepackage{newlfont}
\usepackage{epsfig}
\usepackage{amssymb}
\usepackage{amsmath}
\usepackage{amsthm}
\usepackage{graphicx}
\usepackage{makeidx}
\usepackage[mathscr]{euscript}
\theoremstyle{plain}
\newtheorem{thm}{Theorem}[section]
\newtheorem{cor}[thm]{Corollary}
\newtheorem{lem}[thm]{Lemma}

\theoremstyle{definition}

\theoremstyle{remark} \tolerance=10000 \hbadness=10000
\def \ni{\noindent}
\author{Jeepamol J. Palathingal \footnote{Email : jeepamoljp@gmail.com}\\ Department of Mathematics\\
	PM Government College\\ \vspace{0.3cm} Chalakudy-680722, Kerala, India.\\
	Aparna Lakshmanan S.\footnote{E-mail :
		aparnaren@gmail.com, aparnals@cusat.ac.in}\\
	Department of Mathematics\\Cochin University of Science and Technology\\ \vspace{0.3cm} Cochin-682022, Kerala, India.\\
	Greg Markowsky \footnote{E-mail :
		greg.markowsky@monash.edu} \\
	School of Mathematics \\
	Monash University \\	Clayton, VIC Australia}

\title{\textbf{The transformation of edge-regular and pseudo strongly regular graphs under graph operations}}
\date{}
\begin{document}
	\maketitle
	\begin{abstract}	
	The graph $G$ is said to be  strongly regular with parameters $(n,k,\lambda,\mu)$ if the following conditions hold: (1) each vertex has $k$ neighbours;
		(2) any two adjacent vertices of $G$ have $\lambda$ common neighbours;
		(3) any two non-adjacent vertices of $G$ have $\mu$ common neighbours. In this paper we study two weaker notions  of strongly regular graphs. A graph satisfying the conditions $(1)$ and $(2)$ is called an   edge-regular graph  with parameters $(n,k,\lambda)$. We call a graph satisfying the conditions $(1)$ and $(3)$ a pseudo strongly regular graph with parameters $(n,k,\mu)$. In this paper we study the impact of various graph operations on edge regular graphs and pseudo strongly regular graphs.
		\ni\line(1,0){360}
		
		\ni {\bf Keywords:} Strongly regular graph, edge-regular graph, pseudo strongly regular graph, graph operations.
		
		\ni\line(1,0){360}
		
	\end{abstract}
	\section{Introduction}

	The graph $G$ is said to be {\it strongly regular} \cite{Bap} with parameters $(n,k,\lambda,\mu)$ if the following conditions hold:\\
	\noindent (1) each vertex has $k$ neighbours;\\
	(2) any two adjacent vertices of $G$ have $\lambda$ common neighbours;\\
	(3) any two non-adjacent vertices of $G$ have $\mu$ common neighbours.

In literature there are several classes of graphs containing the strongly regular graphs, but which are defined by a weakening of these conditions. A graph satisfying the conditions $(1)$ and $(2)$ is called an  {\it edge-regular graph} \cite{L} with parameters $(n,k,\lambda)$. This is a relatively weak condition, and there are numerous examples. Recent research has focussed on the clique structure of such graphs (\cite{greaves2018edge,soicher2015cliques}). A graph such that the number of common neighbours of any pair of vertices can be one of just two values is called a  {\it Deza graph} with parameters $(n,\lambda,\mu)$; this family has attracted considerable recent interest (\cite{deza2003generalization, erickson1999deza, goryainov2021deza}). A graph satisfying the conditions $(1)$ and $(3)$ has been mentioned in \cite{L}, but to our knowledge has not been given any specific name. Since it is a weaker version of strongly regular graph, we will refer to it as a {\it pseudo strongly regular graph}. It is evident that strongly regular graphs are edge-regular, Deza, and pseudo strongly regular, but the converse does not hold. In fact, the class of strongly regular graphs is far more restricted than that of others, with a rich structure that has attracted numerous researchers. For excellent overviews on this topic, see \cite{Bro, cam,godroy}.

In this paper we study how applying graph operations transforms edge-regular, pseudo strongly regular, and strongly regular graphs. For that we consider the following graph operations: the complement of a graph, the cartesian product, the direct product, the composition, the strong product, the join, the line graph, the subdivision graph, and the semi-total point graph. The next section gives a few examples of edge-regular and pseudo strongly regular graphs, and then each of the subsequent sections will be devoted to one of the operations, and will contain the required definitions.
	
	All graph theoretic notations and terminology we use are standard and can be found, for instance, in \cite{Bal}. Before beginning our investigation, we isolate some notation for the readers convenience. We will write $u \sim v$ for $u$ adjacent to $v$, $u \nsim v$ for $u$ non-adjacent to $v$ \cite{cv2}, and we denote the number of common neighbors of $u$ and $v$ by $|com\{u,v\}|$.

\section{Examples of edge-regular and pseudo strongly regular graphs} \label{examples}

	\ni There are many examples of edge-regular graphs present in the literature (\cite{Bragan} contains a plethora). We also have
	the following direct \textbf{observations}.
	\begin{enumerate}
		
		\item All $K_3$-free regular graphs are edge-regular.
		\item The graph $C_4 \vee 2K_1$ (see Section \ref{join_section}) is edge-regular with parameters $(6,4,2)$. This is the well-known octahedral graph.
		\item The \textbf{semi-total point graph} $R(G)$ of a graph $G$ is obtained from $G$ by adding a new vertex corresponding to every edge of $G$, then joining each new vertex to the end vertices of the corresponding edge \cite{Sam}. Take two copies of $R(C_n)$. Let $v_1,v_2,...,v_n$ be the vertices of degree $2$ in the first copy of $R(C_n$) and $v_1',v_2',...,v_n'$ be the vertices of degree $2$ in the second copy of $R(C_n$). Merge $v_i$ with $v_i'$. The new graph is edge-regular with parameters $(3n,4,1)$.
	\end{enumerate}
\begin{center}
\includegraphics[width=.35\textwidth]{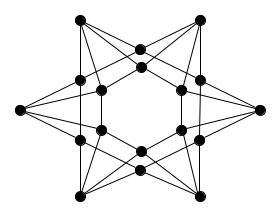}
\end{center}

The class of pseudo strongly regular graphs is not as well documented as that of the edge-regular graphs, but in light of Theorem \ref{11} below the easiest way to construct them is as complements of edge-regular graphs. This does not always lead to interesting examples, for instance the complement of the octahedral graph is the disjoint union of three copies of $K_2$. However, the complement of the graph made from two copies of $R(C_n)$ described above is a more interesting example. It is pseudo strongly regular with parameters $(3n, 3n-5,3n-7)$.

\section{Complement}

If $G$ is a graph, then the {\it complement} of $G$, denoted $G^c$, has the same vertex set as $G$ with adjacency in $G^c$ defined by $x \sim y$ in $G^c$ precisely when $x \nsim y$ in $G$. The following theorem may be considered known, due to its importance in connection with strongly regular graphs (see the corollary below), but is nevertheless included for completeness.

\begin{thm}\label{11}
 $G$  is an edge-regular graph with parameters $(n,k,\lambda)$ if and only if $G^c$ is a pseudo strongly regular graph with parameters $(n,n-k-1,n-2k+\lambda)$.
\end{thm}

\begin{proof}
	
	Consider an edge regular graph $G$ with parameters $(n,k,\lambda)$. Clearly,  $G^c$ is $(n-1-k)$-regular. Now consider two vertices $u$ and $v$ which are not adjacent in $G^c$. Then $u$ and $v$ are adjacent in $G$. Then the number of common neighbours of $u$ and $v$ in $G^c$ is $n-2-(2k-2-\lambda)$. Hence  $G^c$ is pseudo strongly regular with parameters $(n,n-k-1,n-2k+\lambda)$.
	
	For the converse, assume $G^c$ is pseudo strongly regular with parameters $(n,n-k-1,\lambda)$. Clearly $G$ is $k$-regular. Now consider two adjacent vertices $u$ and $v$ in $G$. The number of common neighbours of $u$ and $v$ in $G$ is $n-2-(2(n-k-1)-(n-2k+\lambda))$. Hence $G^c$ is an edge-regular graph with parameters $(n,k,\lambda)$.
\end{proof}

\begin{cor}
If $G$ is a strongly regular graph with parameters $(n,k, \lambda, \mu)$ then $G^c$ is strongly regular with parameters $(n,n-k-1, n-2-2k+\mu, n-2k+\lambda)$.
\end{cor}

\section{Cartesian product}	
	The \textbf{Cartesian product}  \cite{imrich2008topics, W} of $G$ and $H$, written $G\square H$, is the graph with vertex set $V(G_1)\times V(G_2)$ and with adjacency defined by having $(u,v)$ adjacent to $(u',v')$ precisely when \\
	(1)$u=u'$ and $vv'\in E(H)$, or\\
	(2) $v=v'$ and $uu'\in E(G)$.
	
The question of whether the Cartesian product of two edge-regular graphs is itself edge-regular has a rather pleasing answer, as follows.
	
\begin{thm}\label{12}
Let $G_1=(n_1,k_1,\lambda_1)$ and $G_2=(n_2,k_2,\lambda_2)$ be two edge-regular graphs. $G_1\square G_2$ is an edge-regular graph with parameters $(n_1n_2, k_1+k_2,\lambda)$ if and only if $\lambda_1=\lambda_2=\lambda$.
\end{thm}
\begin{proof}
By the definition of cross product  $G_1\square G_2$ is $(k_1+k_2)$-regular. Now consider two adjacent vertices in  $G_1\square G_2$. Then either the vertices are of the form $(u_i,v_l), (u_i,v_m)$ where $v_l$ and $v_m$ are adjacent in $G_2$ or the vertices are of the form $(u_i,v_l), (u_j,v_l)$ where $u_i$ and $u_j$ are adjacent in $G_1$. For the first case the number of common vertices of $(u_i,v_l)$ and $ (u_i,v_m)$ is $\lambda_2$ while for the second case the number of common vertices of $(u_i,v_l)$ and $ (u_j,v_l)$ is $\lambda_1$. Hence $G_1\square G_2$ is edge-regular if and only if $\lambda_1=\lambda_2$.
\end{proof}

We can also ask when this Cartesian product is pseudo strongly regular, but the answer is not as interesting. We omit the straightforward proof.

\begin{thm}
If $G_1=(n_1,k_1,\mu_1)$ and $G_2=(n_2,k_2,\mu_2)$ are two edge-regular graphs, then $G_1\square G_2$ is pseudo strongly regular if and only if one of the graphs is complete and the other is totally disconnected.
\end{thm}

\section{Direct product}	
	
	The \textbf{direct(or tensor or Kronecker) product} \cite{hammack2011handbook} of $G_1$ and $G_2$, written $G_1\times G_2$, is the graph with vertex set $V(G_1)\times V(G_2)$ and with adjacency specified by having $(u_1,v_1)$ adjacent to $(u_2,v_2)$ precisely when $u_1$ is adjacent to $v_1$ in $G_1$ and $u_2$ is adjacent to $v_2$ in $G_2$. In contrast to the Cartesian product, the direct product of two edge-regular graphs is always edge-regular.
	
	\begin{thm}\label{13}
	Let $G_1=(n_1,k_1,\lambda_1)$ and $G_2=(n_2,k_2,\lambda_2)$ be two edge-regular graphs. Then $G_1\times G_2$ is an edge-regular graph with parameters $(n_1n_2,k_1k_2,\lambda_1\lambda_2)$.
\end{thm}
\begin{proof}
	The degree of an arbitrary vertex in $G_1\times G_2$  is $k_1k_2$. If $(u_i,v_l)$ and $(u_j,v_m)$ are adjacent in  $G_1\times G_2$ then $u_i$ and $u_j$ are adjacent in $G_1$ and $v_l$ and $v_m$ are adjacent in $G_2$. The number of common vertices of $(u_i,v_l)$ and $(u_j,v_m)$ is $\lambda_1\lambda_2$. Hence $G_1\times G_2$ is an edge-regular graph with parameters $(n_1n_2,k_1k_2,\lambda_1\lambda_2)$.
\end{proof}
	
	Following the definitions, it can be shown that $G_1\times G_2$ is pseudo strongly regular if and only if  $G_1$ and $G_2$ are strongly regular and  $\lambda_2=\mu_2= k_2,\lambda_1=\mu_1=k_1$. However, no such graphs can exist, since we must always have $\lambda \leq k-1$ for a strongly regular graph.
\section{Composition}	
	
	The \textbf{composition (or wreath or lexicographic product)} of $G_1$ and $G_2$, written $G_1[G_2]$, is the graph with vertex set $V(G_1)\times V(G_2)$, with adjacency specified by having $(u_1,v_1)$ adjacent to $(u_2,v_2)$ precisely when $u_1$ is adjacent to $v_1$ in $G_1$ or $u_1=v_1$ and $u_2$ is adjacent to $v_2$ in $G_2$.	The question of the edge-regularity of the composition product of two edge-regular graphs is handled by the following theorem.

\begin{thm}\label{14}
	Let $G_1=(n_1,k_1,\lambda_1)$ and $G_2=(n_2,k_2,\lambda_2)$ be two edge-regular graphs. $G_1[G_2]$ is edge-regular with parameters $(n_1n_2,k_1n_2+k_2,\lambda)$ if and only if $\lambda_1n_2+2\lambda_2= k_1n_2+\lambda_2=\lambda$.
\end{thm}
\begin{proof}
	An arbitrary  vertex $(u_i,v_l)$ in $G_1[G_2]$ is adjacent  to $k_1n_2+k_2$ vertices. Hence $G_1[G_2]$ is $(k_1n_2+k_2)$-regular. Now consider two adjacent vertices in  $G_1[G_2]$. Then either the vertices are of the form $(u_i,v_l), (u_j,v_m)$, where $u_i$ and $u_j$ are adjacent in $G_1$, or the vertices are of the form $(u_i,v_l), (u_i,v_m)$ where $v_l$ and $v_m$ are adjacent in $G_2$. For the first case a vertex $(x,y)$ is adjacent to both  $(u_i,v_l)$ and $ (u_j,v_m)$ if and only if one of the following conditions are satisfied. 
		\begin{enumerate}
		\item  $x \sim \{u_i, u_j\}$.
		\item
		$x=u_i, y \sim v_l$.
		\item
	$x=u_j, y\sim v_m$.	
	\end{enumerate}
Combining these conditions we see that in this case the number of common vertices is $\lambda_1n_2+2\lambda_2$. For the second case a vertex $(x,y)$ is adjacent to both  $(u_i,v_l)$ and $ (u_i,v_m)$ if and only if either one of the following conditions are satisfied. 
	 \begin{enumerate}
	 	\item  $x \sim u_i$.
	 	\item
	 	$x=u_i, y \sim \{v_l,v_m\}$.
	 		 \end{enumerate}
 		 In this case the number of common vertices is $k_1n_2+\lambda_2$.
 		 Hence $G_1\square G_2$ is edge-regular if and only if  $\lambda_1n_2+2\lambda_2= k_1n_2+\lambda_2$.
\end{proof}

\begin{thm}\label{13}
	Let $G_1=(n_1,k_1,\mu_1)$ and $G_2=(n_2,k_2,\mu_2)$ be two pseudo strongly regular graphs then $G_1[G_2]$ is pseudo strongly regular if and only if $\mu_1 = k_1$ and $\mu_2=0$ (in this case, $G_2$ is a disjoint union of complete graphs).
\end{thm}
\begin{proof}
	
	\ni Consider two non-adjacent vertices $x$ and $y$ in $G_1\times G_2$. Since $x$ and $y$ are non adjacent the following cases are possible.
	\begin{enumerate}
		\item  $x=(u_1,v_1), y= (u_2, v_2)$, $u_1\nsim u_2$.
		\item
		$x=(u_1,v_1), y= (u_2, v_2),u_1=u_2, v_1\nsim v_2$.
		
	\end{enumerate}
	For the first case a vertex $(p,q)$ is adjacent to $(u_1,v_1)$ and $ (u_2, v_2)$ if and only if $p\sim \{u_1,u_2\}$. Then $|com\{x,y\}|= \mu_1 n_2$. Now for the second case  $(p,q)$ is adjacent to $(u_1,v_1)$ and $ (u_1, v_2)$ if and only if either $p\sim u_1$ or $p=u_1$ and $q \sim \{v_1,v_2\}$. In this case, $|com\{x,y\}|= k_1n_2+\mu_2$.  Hence $G_1[G_2]$ is pseudo strongly regular if and only if $\mu_1 n_2=k_1n_2+\mu_2$. Since there exists no graphs with $\mu_1>k_1$, $G_1[G_2]$ is pseudo strongly regular if and only if $\mu_1 = k_1$ and $\mu_2=0$.
\end{proof}

\section{Strong product}	

	The \textbf{Strong (or normal) product} \cite{hammack2011handbook}, denoted by $G1 \boxtimes G2$ is the union of cartesian and direct prodcts, ie; $(G_1\square G_2) \cup (G_1\times G_2)$. The edge-regularity of the composition product of two edge-regular graphs is determined by the following theorem.

	\begin{thm}\label{15}
	Let $G_1=(n_1,k_1,\lambda_1)$ and $G_2=(n_2,k_2,\lambda_2)$ be two edge-regular graphs. Then $G_1 \boxtimes G_2$ is an edge regular graph with parameters $(n_1n_2,k_1+k_2+k_1k_2, \lambda)$ if and only if $\lambda_2+k_1\lambda_2+2k_1= \lambda_1+k_2\lambda_1+2k_2= \lambda_1\lambda_2+2\lambda_2+2\lambda_1+2=\lambda$.
\end{thm}
\begin{proof}
	From the definition it is clear that $G_1 \boxtimes G_2$ is $(k_1+k_2+k_1k_2)$-regular. When we consider two adjacent vertices in  $G_1\boxtimes G_2$ the following three cases are possible.
	
	\begin{itemize}
	    \item[$(i)$] The vertices are of the form $(u_i,v_l), (u_i,v_m)$, where $v_l$ and $v_m$ are adjacent in $G_2$
	    \item[$(ii)$] The vertices are of the form $(u_i,v_l), (u_j,v_l)$, where $u_i$ and $u_j$ are adjacent in $G_1$
	
	    \item[$(iii)$] The vertices are of the form $(u_i,v_l), (u_j,v_m)$, where $u_i$ and $u_j$ are adjacent in $G_1$ and $v_l$ and $v_m$ are adjacent in $G_2$.
	\end{itemize}
	
	For the Case $(i)$, a vertex $(x,y)$ is adjacent to both  $(u_i,v_l)$ and $ (u_i,v_m)$ if and only if either one of the following conditions are satisfied.
	\begin{enumerate}
		\item  $x=u_i, y\sim \{v_l,v_m\}$.
		\item
		$x \sim u_i, y\sim \{v_l,v_m\}$
		\item
		$x \sim u_i, y = v_l$.
		\item
		$x \sim u_i, y = v_m$.	
	\end{enumerate}
	In this case the number of common vertices is $\lambda_2+k_1\lambda_2+2k_1$. In the same way, for Case $(ii)$, we can prove that the number of common vertices is $\lambda_1+k_2\lambda_1+2k_2$. Now, for the Case $(iii)$, a vertex $(x,y)$ is adjacent to both  $(u_i,v_l)$ and $ (u_j,v_m)$ if and only if either one of the following conditions are satisfied.
	\begin{enumerate}
		\item  $x \sim \{u_i,u_j\}, y\sim \{v_l,v_m\}$.
		\item  $x=u_i, y\sim \{v_l,v_m\}$.
		\item  $x=u_j, y\sim \{v_l,v_m\}$.
	
		\item
		$x \sim \{u_i,u_j\}, y = v_l$.
		\item
		$x \sim \{u_i,u_j\}, y = v_m$.
		\item
		$(x,y)=(u_i,v_m)$ or $(x,y)=(u_j,v_l)$	
	\end{enumerate}
	 In this case the number of common vertices is $\lambda_1\lambda_2+2\lambda_2+2\lambda_1+2$. Hence $G_1 \boxtimes G_2$ is edge-regular if and only if $\lambda_2+k_1\lambda_2+2k_1= \lambda_1+k_2\lambda_1+2k_2= \lambda_1\lambda_2+2\lambda_2+2\lambda_1+2$.
\end{proof}

	It is straightforward to verify that $G_1\boxtimes G_2$ is pseudo strongly regular if and only if $G_1$ and $G_2$ are strongly regular and $\mu_2= 1+k_2,\mu_1=k_1$. But since we must always have $\mu_2 \leq k_2$, such graphs do not exist.

\section{Join} \label{join_section}

	Let $G_1$ and $G_2$ be vertex-disjoint graphs. Then the \textbf{join}, $G_1\vee G_2$, of $G_1$ and $G_2$ is the supergraph of the vertex disjoint union of $G_1$ and $G_2$ in which each vertex of $G_1$ is adjacent to every vertex of $G_2$ \cite{Bal}. The following theorem addresses whether the join of two edge-regular graphs is again edge-regular.

	\begin{thm}\label{26}
Let $G_1=(n_1,k_1,\lambda_1)$ and $G_2=(n_2,k_2,\lambda_2)$ be two edge-regular graphs. Then $G_1\vee G_2$ is an edge regular graphs with parameters $(n_1+n_2,k,\lambda)$ if and only if $k_1+n_2=k_2+n_1=k$ and $\lambda_1+n_2=\lambda_2+n_1=k_1+k_2=\lambda$.
\end{thm}
\begin{proof}
By the definition $G_1\vee G_2$ is regular if and only if $k_1+n_2= k_2+n_1$. Similarly to the previous theorem when we consider two adjacent vertices $u$ and $v$ of $G_1\vee G_2$ the following three cases arise.  Either $\{u,v\} \in V(G_1)$ or $\{u,v\} \in V(G_2)$  or $u \in V(G_1)$ (or $V(G_2)$) and $v \in V(G_2)$ (or $V(G_1)$). For the first case the number of common vertices  is $\lambda_1+ n_2$. For the second case the number of common vertices  is $\lambda_2+n_1$. For the third case the number of common vertices  is $k_1+k_2$. Hence $G_1\vee G_2$ is edge-regular if and only if $\lambda_1+n_2=\lambda_2+n_1=k_1+k_2$ and $k_1+n_2=k_2+n_1$.
\end{proof}

	Similarly, we can determine whether the join of two pseudo strongly regular graphs is again pseudo strongly regular.

\begin{thm}
	Let $G_1=(n_1,k_1,\mu_1)$ and $G_2=(n_2,k_2,\mu_2)$ be two pseudo strongly regular graphs. Then $G_1\vee G_2$ is pseudo strongly regular with parameters $(n_1+n_2,k,\mu)$ if and only if $k_1+n_2=k_2+n_1=k$ and  $\mu_1+n_2=\mu_2+n_1=\mu$.
\end{thm}
\begin{proof}
As in Theorem \ref{26}, $G_1\vee G_2$ is regular if and only if $k_1+n_2= k_2+n_1$. As in the previous theorem when we consider two non-adjacent vertices $u$ and $v$ belong to $G_1\vee G_2$, either $\{u,v\} \in V(G_1)$ or  $\{u,v\} \in V(G_2)$ . For the first case $|com\{u,v\}|= \mu_1+n_2$. For the second case $|com\{u,v\}|= \mu_2 +n_1$.  Hence $G_1\vee G_2$ is edge-regular if and only if $\mu_1+n_2=\mu_2+n_1$ and $k_1+n_2=k_2+n_1$.
\end{proof}

\section{Line graph}
	
	The line graph $L(G)$ of a graph $G$ has the edges of $G$ as its vertices, and two vertices corresponging to distinct edges of $G$ are adjacent in $L(G)$ if the two edges are incident in $G$ \cite{Pri}. Line graphs are important in algebraic graph theory; for instance, \cite{godroy} devotes an entire chapter to them. The line graph of an edge-regular graph can only be edge-regular in the extreme cases $\lambda=0$ or $\lambda=n-2$, as we now show.

	\begin{thm}\label{27}
Let $G=(n,k,\lambda)$ be an edge-regular graph. Then $L(G)$ is edge-regular if and only if either $G$ is $K_3$-free or $G$ is a complete graph.
\end{thm}
\begin{proof}
	Since $G$ is $k$-regular, $L(G)$ is $(2k-2)$-regular. Now consider two adjacent vertices $u$ and $v$ in $L(G)$. Since they are adjacent in $L(G)$ the corresponding edges $e_1$ and $e_2$  are incident in $G$. The number of common vertices of $u$ and $v$ is same as the number of edges which are incident on both $e_1$ and $e_2$ in $G$. If $e_1$ and $e_2$ induce a $K_3$ in $G$ then it is $k-1$, otherwise it is $k-2$. Hence $L(G)$ is edge-regular if and only if either $G$ is $K_3$-free or $G$ is $K_{1,2}$-free.	
\end{proof}
	
	Remarkably, the property of pseudo strong regularity behaves better than edge regularity with regards to line graphs. Before showing this, we need a preperatory lemma.
	
	\begin{lem} \label{37}
	Suppose $H$ is a pseudo strongly regular graph with parameters $(m,k,\mu)$ and $H\cong L(G)$. Then $\mu \leq 4$.
\end{lem}
\begin{proof}
Consider two non adjacent vertices $u$ and $v$ in $H$. Let $e_1$ and $e_2$ be the corresponding edges in $G$. Then $|com\{u,v\}| $ is same as the number of edges incident on both $e_1$ and $e_2$ in $G$. Since $e_1$ and $e_2$ cannot have a common vertex in $G$, the number of such edges is at the most $4$. Hence $\mu \leq 4 $.
\end{proof}
\ni\textbf{Note:}  The line graph of a pseudo strongly regular graph is pseudo strongly regular with $\mu=0$ if and only if  $L(G)$ is the union of triangles or isolated vertices (equivalently, $G$ is union of triangles or edges). Note that, though $L(K_{1,3})$ is $K_3$, which is pseudo strongly regular, $K_{1,3}$ is not even regular. In general, there are graphs which are not pseudo strongly regular, but whose line graph is pseudo strongly regular. For example, $G = K_{1,n}$ is not regular and hence not pseudo strongly regular, but $L(K_{1,n})$ is strongly regular and hence also pseudo strongly regular.
\begin{lem}
	There does not exist any pseudo strongly regular graph $G$ whose $L(G)$ is pseudo strongly regular with parameter $\mu=3$.
\end{lem}
\begin{proof}
On the contrary, assume that $G$ is a graph such that $L(G)$ is pseudo strongly regular with parameter $\mu=3$. By Lemma \ref{37}, the graph $G$ is $K_4$-free and any two non-adjacent edges of $G$ belong to a diamond. Let $u_1,u_2,u_3$ and $u_4$ induce a diamond in $G$, where $u_1$ and $u_4$ are the non-adjacent pair. Since $G$ is regular, there exists at least one vertex adjacent to $u_1$, say $u_5$. Since the edges $u_1u_5$ and $u_3u_4$ are not adjacent, the vertices $u_1,u_5,u_3$ and $u_4$  induce a diamond. Since the graph is $K_4$-free, the only possibility is $u_5$ is adjacent to both $u_3$ and $u_4$.  Now, consider the non-adjacent pair of edges $u_1u_5$ and $u_2u_4$. By the same argument either $u_2u_5$ or  $u_1u_4$ is an edge in $G$ and in either case the graph contains a $K_4$, which is a contradiction. The lemma follows.
\end{proof}
\begin{thm}
	Let $G=(n,k,\mu)$ be a pseudo strongly regular graph. Then $L(G)$ is pseudo strongly regular with parameter,
	\begin{enumerate}
		\item  $\mu=4$ if and only if $G \cong K_n$, where $n\geq 4$;
		\item  $\mu=2$ if and only if $G$ is (diamond, $K_4$)-free and any two edges belong to a  $C_4$;
		\item  $\mu=1$ if and only if $G$ is (diamond, $K_4, C_4$)-free and any two edges belong to a $P_4$.
	\end{enumerate}
\end{thm}
\begin{proof}
By Lemma \ref{37}, $L(G)$ is pseudo strongly regular with $\mu=4$ if and only if any two non-adjacent edges belong to a $K_4$ and $\mu= 2,1$ if and only if any two non-adjacent edges of $G$ belong to a $C_4$, $P_4$ respectively. The result follows.
\end{proof}

\section{Subdivision graph}	
	
	If $uv$ is an edge of $G$, then $uv$ may be {\it subdivided} by introducing a new vertex $w$ and then replacing $uv$ by the edges $uw$ and $wv$. If every edge of $G$ is subdivided then the resulting graph is the \textbf{subdivision graph}, denoted by $S(G)$ \cite{Har}.
	
\begin{thm}
Let $G=(n,k,\lambda)$ be an edge-regular graph. $S(G)$ is  edge-regular if and only if $G$ is the disjoint union of cycles.	
\end{thm}
\begin{proof}
Since $S(G)$ is $K_3$-free the number of common vertices of two adjacent vertices is zero. But the new vertices have degree $2$. Therefore $S(G)$ is edge-regular if and only if $G$ is a $2$-regular graph. Hence $S(G)$ is edge-regular if and only if $G$ is the disjoint union of cycles.
\end{proof}

On the other hand, it is not hard to see that there does not exists any graph $G$ such that $S(G)$ is pseudo strongly regular. As in the proof of the previous theorem, $S(G)$ is regular if and only if $G$ is 2-regular, in other words $G$ is the disjoint union of cycles. Among these graphs the only pseudo strongly regular graph is the $C_5$. But odd cycles are forbidden in $S(G)$, and hence there does not exist any graph $G$ such that $S(G)$ is pseudo strongly regular.
	
\section{Concluding remarks} \label{semi_point_graph}	
	
We have given a number of examples of graph operations and their effect on the classes of edge-regular and pseudo strongly regular graphs, however there are a seemingly limitless supply of other operations that we have not touched upon. Doubtless there are many other results analogous to ours waiting to be uncovered.

\if2
\section{Pseudo strongly regular graphs}
 The regularity of the following operators are already proved in session $2$. So here in the proof we are looking  only on the  $|com\{u,v\}|$.\\
\textbf{Observations}\\
\begin{enumerate}
	\item
	\item

	\item

	\item

\end{enumerate}
\fi


\newcommand{\etalchar}[1]{$^{#1}$}
\begin{thebibliography}{GHKS19}

\bibitem[DH03]{deza2003generalization}
Michel Deza and Tayuan Huang.
\newblock A generalization of strongly regular graphs.
\newblock {\em Southeast Asian Bulletin of Mathematics}, 26(2):193--201, 2003.

\bibitem[EFH{\etalchar{+}}99]{erickson1999deza}
M~Erickson, S~Fernando, WH~Haemers, D~Hardy, and J~Hemmeter.
\newblock Deza graphs: A generalization of strongly regular graph.
\newblock {\em Journal of Combinatorial Designs}, 7(6):395--405, 1999.

\bibitem[GGK14]{gavrilyuk2014vertex}
AL~Gavrilyuk, SV~Goryainov, and VV~Kabanov.
\newblock On the vertex connectivity of deza graphs.
\newblock {\em Proceedings of the Steklov Institute of Mathematics},
  285(1):68--77, 2014.

\bibitem[GHKS19]{goryainov2019deza}
Sergey Goryainov, Willem~H Haemers, Vladislav~V Kabanov, and Leonid Shalaginov.
\newblock Deza graphs with parameters and.
\newblock {\em Journal of Combinatorial Designs}, 27(3):188--202, 2019.

\bibitem[GK18]{greaves2018edge}
Gary~RW Greaves and Jack~H Koolen.
\newblock Edge-regular graphs with regular cliques.
\newblock {\em European Journal of Combinatorics}, 71:194--201, 2018.

\bibitem[GS21]{goryainov2021deza}
Sergey Goryainov and Leonid~V Shalaginov.
\newblock Deza graphs: a survey and new results.
\newblock {\em arXiv preprint arXiv:2103.00228}, 2021.

\bibitem[HIK{\etalchar{+}}11]{hammack2011handbook}
Richard~H Hammack, Wilfried Imrich, Sandi Klav{\v{z}}ar, Wilfried Imrich, and
  Sandi Klav{\v{z}}ar.
\newblock {\em Handbook of product graphs}, volume~2.
\newblock CRC press Boca Raton, 2011.

\bibitem[IKR08]{imrich2008topics}
Wilfried Imrich, Sandi Klavzar, and Douglas~F Rall.
\newblock {\em Topics in graph theory: Graphs and their Cartesian product}.
\newblock CRC Press, 2008.

\bibitem[KKS21]{kabanov2021generalised}
Vladislav~V Kabanov, Elena~V Konstantinova, and Leonid Shalaginov.
\newblock Generalised dual seidel switching and deza graphs with strongly
  regular children.
\newblock {\em Discrete Mathematics}, 344(3):112238, 2021.

\bibitem[KS20]{kabanov2020deza}
Vladislav~V Kabanov and Leonid Shalaginov.
\newblock Deza graphs with parameters (v, k, k- 2, a).
\newblock {\em Journal of Combinatorial Designs}, 28(9):658--669, 2020.

\bibitem[Soi15]{soicher2015cliques}
Leonard~H Soicher.
\newblock On cliques in edge-regular graphs.
\newblock {\em Journal of Algebra}, 421:260--267, 2015.

\end{thebibliography}


\begin{thebibliography}{}
	\bibitem{Bal} Balakrishnan, R. and K. Ranganathan, A textbook of graph theory, Springer, 1999.
	\bibitem{Bap} Bapat, R.B., Graphs and Matrices, Springer, 2010.
	\bibitem{Bragan} Bragan, K., Topics in Edge Regular Graphs (Doctoral dissertation), 2014.
	
	\bibitem{Bro}  Brouwer, A. and Van Maldeghem, H., Strongly regular graphs, https://homepages.cwi.nl/\~aeb/math/srg/rk3/srgw.pdf.
	
	\bibitem{cam} Cameron, P., Strongly regular graphs, {\it Topics in Algebraic Graph Theory}, 102, 2004, 203-221.
	

	
	\bibitem{cv2} Cvetcovic, D., Rowlinson, P., and Simic, S., An Introduction to the Theory of Graph Spectra, London Mathematical Society Students Texts, 2010.
	
\bibitem{deza2003generalization}
Deza, M. and Huang, T.,
\newblock A generalization of strongly regular graphs.
\newblock {\em Southeast Asian Bulletin of Mathematics}, 26(2), p. 193--201, 2003.

\bibitem{erickson1999deza}
Erickson, M., Fernando, S., Haemers, W., Hardy, D., and Hemmeter, J.,
\newblock Deza graphs: A generalization of strongly regular graph.
\newblock {\em Journal of Combinatorial Designs}, 7(6), p. 395--405, 1999.

	
	
	\bibitem{godroy} Godsil, C. and Royle, G., Algebraic graph theory. Vol. 207. Springer Science and Business Media, 2001.


\bibitem{goryainov2021deza}
Goryainov, S. and Shalaginov, L.,
\newblock Deza graphs: a survey and new results.
\newblock {\em arXiv preprint arXiv:2103.00228}, 2021.

\bibitem{greaves2018edge}
Greaves, G. and Koolen, J.
\newblock Edge-regular graphs with regular cliques.
\newblock {\em European Journal of Combinatorics}, 71:194--201, 2018.


\bibitem{hammack2011handbook}
Hammack, R., Imrich, W., Klav{\v{z}}ar, S., Imrich, W., and
  Klav{\v{z}}ar, S.,
\newblock Handbook of product graphs,
\newblock CRC press Boca Raton, 2011.

\bibitem{Har} Harary, F., Graph Theory, Narora publishing House, 1988.

\bibitem{imrich2008topics}
Imrich, W., Klavzar, S., and Rall, D.
\newblock Topics in graph theory: Graphs and their Cartesian product,
\newblock CRC Press, 2008.


	
\bibitem{Ka} Kavitha, K. and David, N.G., Dominator coloring of some classes of graphs, {\it International Journal of Mathematical Archive}, 3(11), p. 3954 - 3957, 2012.
	
\bibitem{Pri} Prisner, E., "Graph Dynamics", Longman, 1995.
	
	\bibitem{Le1} Le, V. B., Gallai graphs and Anti-Gallai Graphs, {\it Discrete Mathematics}, 159, p. 179 - 189, 1996.
	

	
	\bibitem{L}  Beineke, L., Wilson, R., and Cameron, P.,Topics in Algebraic Graph Theory: 102 (Encyclopedia of Mathematics and its Applications, Series Number 102), Cambridge University Press, 2004.
	
	
	
	\bibitem{Sam} Sampathkumar, E. and Chikkodimath, S.B., The Semitotal graphs of a graph-II, {\it J. Karnatak Univ. Sci}, 18, p. 281 - 284, 1973.

	\bibitem{W} West, B., Introduction to Graph Theory, Prentice-Hall, 1999.

\bibitem{soicher2015cliques}
Soicher, L.,
\newblock On cliques in edge-regular graphs.
\newblock {\em Journal of Algebra}, 421:260--267, 2015.

\end{thebibliography}
\end{document}